\documentclass[leqno,10pt]{article}
\usepackage{latexsym}
\usepackage[centertags]{amsmath}
\usepackage{amssymb,amsthm,amsfonts}
\usepackage{color}

\newtheorem{theorem}{Theorem}
\newtheorem{lemma}[theorem]{Lemma}
\newtheorem{prop}[theorem]{Proposition}

\newtheorem{defn}[theorem]{Definition}
\theoremstyle{remark}
\newtheorem*{rmk}{Remark}

\numberwithin{theorem}{section}

\newcommand{\bydef}{:=}

\newcommand{\charf}{\mbox{{\text 1}\kern-.24em {\text l}}}

\newcommand{\del}{\partial}
\newcommand{\dist}{{\mbox{dist}}}
\renewcommand{\div}{\operatorname{div}}

\newcommand{\eps}{\varepsilon}

\renewcommand{\geq}{\geqslant}

\newcommand{\lapi}{\operatorname{{\Delta_\infty}}}
\newcommand{\lapih}{{\Delta_{\infty}^h}}
\renewcommand{\leq}{\leqslant}

\newcommand{\oo}[1]{\frac{1}{#1}}
\newcommand{\openR}{{{\text I}\kern-.22em {\text R}}}

\newcommand{\R}{{\mathbb R}}
\newcommand{\Rd}{{\R^d}}

\numberwithin{equation}{section}
\begin{document}

\title
{\bf Degenerate homogeneous parabolic equations \\ associated with the
infinity-Laplacian
}

\author{Manuel Portilheiro\footnote{\texttt{manuel.portilheiro@uam.es}}\\[4pt]
Juan Luis V\'azquez\footnote{\texttt{juanluis.vazquez@uam.es}}}

\date{}
\maketitle

\begin{abstract}
We prove existence and uniqueness of viscosity solutions to the degenerate parabolic problem $u_t = \lapih u$ where $\lapih$ is the $h$-homogeneous operator associated with the infinity-Laplacian, $\lapih u = |Du|^{h-3} \langle D^2uDu,Du\rangle$. We also derive the asymptotic behaviour of $u$ for the problem posed in the whole space and for the Dirichlet problem with zero boundary conditions.
\end{abstract}

\allowdisplaybreaks

\section{Introduction}\label{S:intro}

In this paper we study the following class of
degenerate parabolic equations for $u:Q\to\R$, $u=u(x,t)$ with a
parameter $h > 1$:
\begin{align}
     u_t - \lapih u &= 0, \quad\text{in $Q$}, \label{E:main} \\
     u &= g, \quad\text{on $\Gamma$}, \label{E:main-bc}
\end{align}
 where $\lapih$ denotes the
$h$-homogeneous degenerate operator,
\[
    \lapih u \bydef |Du|^{h-3} \sum_{i,j=1}^d u_{x_ix_j} u_{x_i} u_{x_j}.
\]
Note the $h$ stands for the homogeneity degree of the operator.
The equation is posed for $(x,t)\in Q \bydef \Omega\times(0,T)$,
where  $\Omega \subset \R^d$ is open (not necessarily bounded), $0<T\leq \infty$, \normalcolor $\Gamma = \del_p Q$ the parabolic boundary of $Q$ and
$g:\Gamma\to\R$ is a continuous function.

Equation \eqref{E:main} has been studied in two particular cases,
$h=1$ and $h=3$, which correspond to the two common definitions of
the infinity-Laplacian. For $h=1$ we find the usual 1-homogeneous
infinity-Laplacian operator, $\lapi u = |Du|^{-2}\sum_{i,j}
u_{x_ix_j} u_{x_i} u_{x_j}$, so that for other $h$ we get the formula
$\lapih u = |Du|^{h-1} (\lapi u)$.
The infinity-Laplacian operator has received a lot of attention in
the last decade, notably due to its application to image processing.
Most of the interesting properties of the
$\infty$-Laplace equation can be found in Crandall's enjoyable
paper \cite{C} (the long list of references therein covers most of
what we avoid listing here). Regarding the evolution of the
infinity-Laplacian, the above two cases ($h=1,3$) have been studied
in a number of works. The 1-homogeneous version has been considered
by Wu, Juutinen-Kawhol and Barron-Evans-Jensen (\cite{W,JK,BEJ}; see
\cite{JK} for additional references in applications), whereas the
3-homogeneous equation has been treated by Crandall-Wang,
Akagi-Suzuki, Akagi-Juutinen-Kajikiya and Laren\c{c}ot-Stinner
(\cite{CW,AS,AJK,LS}).

We think that equation \eqref{E:main}, and its generalization
\eqref{E:gen-main}, apart from their applications, have intrinsic
interest, as they are degenerate parabolic, singular when $h<3$, do
not correspond to a geometric equation, are not variational and are
not in divergence form (with the exception of the 1-homogeneous
infinity-heat equation in two space dimensions, \cite{DGIMR}). They
constitute a class of equations with particular difficulties and
properties.

It is also worthwhile pointing out that radial solutions of
\eqref{E:main} (or if we consider the equation in one space
dimension) correspond to solutions of the $p$-Laplace evolution
equation
\[
	u_t = \div \left( |Du|^{p-2}Du \right)
\]
in one space dimension with $h = p -1$. In fact, a comparison between
the radial solutions in Section~\ref{S:special-sol} and solutions of
the $p$-Laplace evolution make this evident.

Here we consider all cases $h>1$. Since the results and techniques
for $h=1$ are slightly different, we will not consider this case
below and will treat it in a separate paper.  For the case $h=3$
we extend slightly both the existence and the asymptotic results.
Our existence result allows for an extra source term $H(u)$ (see
\eqref{E:gen-main} bellow) and we do not impose the exterior
sphere condition.

We are specially interested in the asymptotic behaviour of solutions. We make the assumption of
nonnegativity on data and solutions. For bounded domains we do not impose any extra conditions on $\Omega$.
We also study the Cauchy problem posed in the whole $\Rd$, where the decay rate is different, as expected, and we get a precise convergence result to a self-similar solution of Barenblatt type.

\subsection{Main results}
To obtain the asymptotic behaviour on bounded domains it will be
useful to prove existence for a slightly more general equation,
allowing for zero-order terms. Therefore, we consider
\begin{equation}\label{E:gen-main}
    u_t = \lapih u + H(u),
\end{equation}
where $H:\R\to\R$ is a continuous function satisfying $H(0)=0$ and growing at most linearly at infinity, that is, there
exists $M>0$ such that $|H(y)|\leq M|y|$.
\begin{theorem}[Existence]\label{T:exist}
Let $Q\subset\R^{d+1}$, $g:\Gamma\to \R$ be continuous and $H:\R\to\R$ be as above. If either $\Omega$ is bounded or $g$ is bounded,
uniformly continuous in $\Gamma$ and
\[
    \lim_{R\to\infty} \sup_{|x|\geq R, t\in[0,T)} |g(x,t)| = 0,
\]
then there exists a unique continuous viscosity solution $u$ of
\eqref{E:gen-main} with $u=g$ on $\Gamma$ whose modulus of
continuity depends on the modulus of continuity of $g$. If $g$ is
Lipschitz continuous, then so is $u$ and its Lipschitz constant
depends on the Lipschitz constant of $g$.
\end{theorem}
\begin{rmk}
In the case $h=3$ this slightly extends \cite[Theorem~2.5]{AS}, not only regarding the extra source term $H(u)$, but as
we do not impose the exterior sphere condition on the boundary of $\Omega$ and we consider unbounded sets.
\end{rmk}

We want to analyse the large time behaviour of the solutions of \eqref{E:main}--\eqref{E:main-bc} in two cases. The first
is the Cauchy problem, $\Omega = \Rd$, with compactly supported initial condition $g = u_0 \in C_{\mathrm{c}}(\Rd)$.
\begin{equation}\label{E:main-cauchy}\begin{cases}
    u_t = \lapih u, &\text{in $\Rd\times (0,\infty)$}, \\
    u(x,0) = u_0(x).
\end{cases}
\end{equation}
In this case $u$ converges to a radial decaying solution.
\begin{theorem}[Asymptotic behaviour for the Cauchy problem]\label{T:Cauchy}
Let $u$ be the unique viscosity solution of \eqref{E:main-cauchy} with $u_0\in C_{\mathrm{c}}(\Rd)$ not identically
zero, $u_0 \geq 0$. Then, for some $R_* >0$,
\[
    \lim_{t\to\infty} t^{-\oo{2h}} \sup_{x\in\Rd} |u(x,t) - B_{R_*,h}(x,t)| = 0.
\]
where $B_{R,h}$ is the Barenblatt solution, which takes the form
\[
	B_{R,h}(x,t) \bydef c_h\, t^{-\oo{2h}} \left[ R^{\frac{h+1}{h}} - t^{-\frac{h+1}{2h^2}}
        |x|^{\frac{h+1}{h}} \right]^{\frac{h}{h-1}}_+,
\]
as given in \eqref{E:barenblatt}.
\end{theorem}

Next we analyse the Dirichlet problem with homogeneous boundary conditions,
\begin{equation}\label{E:homo-Dirichlet}
\begin{cases}
    u_t - \lapih u =0, &\text{in $Q$}, \\
    u(x,0) = u_0(x), &\text{on $\Omega$}, \\
    u(x,t) = 0, &\text{for $x\in\del\Omega$, $t>0$},
\end{cases}
\end{equation}
where $u_0\in C(\overline{\Omega})$ satisfies $u_0(x) = 0$ for $x\in\del\Omega$, $\Omega$ bounded and
$Q = \Omega\times(0,\infty)$. For nonnegative initial data, we obtain the following.
\begin{theorem}[Asymptotic behaviour for the homogeneous Dirichlet problem]\label{T:asymptotic-Dirichlet}
Suppose $u_0\in C(\overline{\Omega})$, $u \geq 0$, $u_0 \not\equiv 0$ and $u_0(x) = 0$ for $x\in\del\Omega$.
There exists a continuous function $F_\Omega:\Omega\to\R$ such that the solution of \eqref{E:homo-Dirichlet}, $u$ satisfies
\[
    \lim_{t\to\infty} t^{\oo{h-1}} u(x,t) = F_\Omega(x).
\]
\end{theorem}
\begin{rmk}[1]
Naturally, the result holds for constant boundary data with initial condition above (below) this constant, since adding a constant
to $u$ still gives a solution of the problem.
\end{rmk}
\begin{rmk}[2] We will show (Theorem~\ref{T:evp}) that the function $G_\Omega \bydef (h-1)^{\oo{h-1}} F_\Omega$ is a positive solution of
the eigenvalue problem
\[
    -\lapih G =  G.
\]
or
\[
	-\lapi G =  \frac{G}{|DG|^{h-1}}.
\]
This is a variation of the eigenvalue problem for the infinity-Laplacian studied in \cite{J},
\[
	-\lapi G = \lambda G
\]
and a variation of the nonlinear eigenvalue problem considered in \cite{PV}
\[
	-\lapi G = \lambda G^p, \qquad 0<p<1
\]
(in connection with this, \cite{JR} considers ``large solutions'' of this equation with $p>1$).
Note also that the function $(x,t) \mapsto t^{\oo{h-1}} F_\Omega(x)$, which we call the friendly giant, is a generalization
of the solutions of the form \eqref{E:friendly-giant} for the radial case.
\end{rmk}

\noindent{\bf Notations}

We always assume $\Omega \subset \Rd$ is open and denote by $Q$ a solid cylinder of the form $Q\times(t_0,t_1)$,
$t_0 < t_1$. We will also separate the parabolic boundary of $Q$ into a lateral boundary and the base,
$\del_p Q = \del_l Q \cup \del_b Q$, with $\del_l Q = \del\Omega \times (t_0,t_1)$ and $\del_b Q =
\bar{Q}\times\{t_0\}$.

\setcounter{tocdepth}{2}
\tableofcontents
\begin{center}
\rule{42em}{0.1ex}
\end{center}

\section{Viscosity solutions}\label{S:viscosity_solutions}

We start our analysis by defining viscosity solutions for Equation
\eqref{E:main}. Note that since  the equation is singular for
$h\in(1,3)$ at points where the gradient of the function vanishes,
the usual definition of viscosity solution needs to be adapted at
the points of singularity. For the singular case we could use the
definition from \cite{CGG} (see also \cite{JK} for  the case
$h=1$). However, the singularity is removable, and since we are
not considering here $h=1$ we can simply recast Equation
\eqref{E:main} as follows:

We  can write  this equation as
\begin{equation}
u_t + F_h(D^2u,Du) = 0 \quad \mbox{with }  \ F_h(M,p) \bydef
-|p|^{h-3}(Mp)\cdot p
\end{equation}
and $F_h: \mathcal{S}_d \times (\Rd\setminus\{0\}) \to \R$, where
$\mathcal{S}_d$ is the set of $d\times d$ real symmetric matrices.
Since, even for $1<h<3$, $\lim_{|p|\to 0} F_h(M,p) = 0$ for every
$M\in\mathcal{S}_d$, we can take the continuous extension of
$F_h$,
\[
    \overline{F}_h(M,p) \bydef \begin{cases}
        F_h(M,p) & \text{if $p\ne 0$},\\
        0 & \text{if $p=0$},
    \end{cases}
\]
and the usual theory of viscosity solutions can be applied to $\overline{F}_h$. For simplicity we abuse notation and
write $F_h$ to denote its continuous extension $\overline{F}_h$.
\begin{defn}[Viscosity Solutions]\label{D:visc-sol}
An upper [lower] semicontinuous function $u:Q \to \R$ is a \textbf{viscosity subsolution}
[\textbf{supersolution}] of \eqref{E:main} in $Q$ if and only if for every $P_0\in Q$ and every function
$\phi\in C^{2,1}(Q)$ touching $u$ from above [below] at $P_0$, that is,
\begin{equation}\label{C:tfa}\begin{aligned}
    u(P_0) &= \phi(P_0) &&\text{and} \\
    u(P) &< \phi(P)\quad [u(P) > \phi(P)] &&\text{for every $P\in Q\setminus \{P_0\}$},
\end{aligned}\end{equation}
we have
\begin{equation}\label{E:is-visc-sol}
    \phi_t + F_h(D^2\phi,D\phi) \leq 0 \quad  [\phi_t + F_h(D^2\phi,D\phi) \geq0].
\end{equation}
A continuous function $u\in C(Q)$ is a \textbf{viscosity solution} when it satisfies both inequalities.
\end{defn}

Let us also record here the comparison property \cite[Theorem~8.2]{CIL} that directly applies to
\eqref{E:main}--\eqref{E:main-bc} (see also \cite[Theorem~2.3]{AS}).
\begin{theorem}[Comparison Principle]\label{T:comparison}
Let $u$ be an upper semicontinuous viscosity subsolution, and $v$ be a lower semicontinuous viscosity supersolution
of \eqref{E:main} such that $u \leq v$ on $\del_p Q$. Then $u \leq v$ in the whole cylinder $Q$.

Furthermore, if $u$ and $v$ are viscosity solutions of \eqref{E:main}, then
\[
    \max_Q |u-v| \leq \max_{\del_p Q} |u-v|.
\]
\end{theorem}

\begin{rmk}
Equation \eqref{E:main} has a few symmetries which we will explore. Whenever $u(x,t)$ is a classical solution, it
is immediate to check that for any orthogonal $d\times d$ matrix $O$, for any $(x_o,t_o)$ in $\Rd\times\R$, $c\in\R$ and
any $\lambda$ and $s$ positive,
\begin{equation}\label{E:sym}
\begin{aligned}
    v(x,t) =& u\left(O(x-x_o),t-t_o\right) + c \\
    w(x,t) =& \lambda u(sx,\lambda^{h-1}s^{h+1} t) \\
    z(x,t) =& -u(x,t)
\end{aligned}
\end{equation}
are also solutions in an appropriate domain. As usual, the argument can be transposed to viscosity solutions.
\end{rmk}

\section{Special solutions}\label{S:special-sol}

\subsection{Similarity solutions}
We need to obtain certain explicit solutions. Let us first look at similarity solutions of \eqref{E:main}. We
are mainly interested in radial solutions, $u(x,t) = v(r,t)$. It is easy to see that $v$ should solve
\[
    v_t = (v_r)^{h-1} v_{rr} = \oo{h}\left( (v_r)^h\right)_r,
\]
which is a 1-dimensional p-Laplacian equation with $p = h+1$. Similarity solutions for this equation are well
known. In fact, if we take $v$ of the form
\[
    v(r,t) = t^\alpha f(y), \quad \text{with} \quad y = rt^{\beta}.
\]
Then, we get for $f$ the equation
\[
    t^{\alpha-1} \left[ \alpha f(y) + \beta y f'(y) \right]
    = t^{\alpha h + \beta(h+1)} \left(f'(y)\right)^{h-1} f''(y).
\]
Choosing $\alpha = \beta = -\oo{2h}$ we obtain
\[
    \alpha y f(y) = \oo{h} \left( f'(y) \right)^h + C.
\]
If we take $C=0$ we can integrate once more to get (where $C$ is a differenc constant)
\[
    f(y) = \left[ C - \left(\oo{2}\right)^{\oo{h}}\frac{h-1}{h+1} y^{\frac{h+1}{h}} \right]^{\frac{h}{h-1}}_+,
\]
which leads to the usual similarity solution
\begin{equation}\label{E:barenblatt}
    B_{R,h}(x,t) = B_h(x,t) \bydef c_h\, t^{-\oo{2h}} \left[ R^{\frac{h+1}{h}} - t^{-\frac{h+1}{2h^2}}
        |x|^{\frac{h+1}{h}} \right]^{\frac{h}{h-1}}_+,
\end{equation}
where
\[
    c_h = \left(\oo{2}\right)^{\oo{h-1}} \left( \frac{h-1}{h+1} \right)^{\frac{h}{h-1}}.
\]
Note that $R$ denotes the radius of the positivity set of $B_{R,h}$ at time $t=1$.
It is straightforward to prove this defines a viscosity solution.
\begin{prop}\label{P:barenblatt-is-visc}
The function $B_{R,h}$ defined in \eqref{E:barenblatt} is a viscosity solution of \eqref{E:main} in $\Rd\times
(0,\infty)$.
\end{prop}
\begin{proof}
Let us for convenience define $A(r,t) = \left[ R^{\frac{h+1}{h}} - t^{-\frac{h+1}{2h^2}}
r^{\frac{h+1}{h}} \right]_+$. We have
\[\begin{aligned}
    \del_t B_h(x,t) &= -\frac{c_h}{2h} t^{-\frac{2h+1}{2h}} A^{\frac{h}{h-1}} + \frac{c_h(h+1)}{2(h-1)h}
        t^{-\frac{2h(h+1)+1}{2h^2}} r^{\frac{h+1}{h}} A^{\oo{h-1}}, \\
    D B_h(x,t) &= -\frac{c_h(h+1)}{h-1} t^{-\frac{2h+1}{2h^2}} r^{\oo{h}} A^{\oo{h-1}} \frac{x}{|x|}, \\
    D^2 B_h(x,t) &= \frac{c_h(h+1)^2}{h(h-1)^2} t^{-\frac{3h+2}{2h^2}} r^{2/h} A^{-\frac{h-2}{h-1}}
            \frac{x\otimes x}{|x|^2} \\
        &\phantom{=} + \frac{c_h(h+1)}{h-1} t^{-\frac{2h+1}{2h^2}} r^{-\frac{h-1}{h}} A^{\oo{h-1}} \left(
            \frac{h-1}{h} \frac{x\otimes x}{|x|^2} - I \right).
\end{aligned}\]
From this computation we see that the set of points where $B_h$ is not twice differentiable is $W \bydef \left\{
(x,t) \in \Rd\times(0,\infty) \mid A(|x|,t) = 0 \; \text{or} \; x=0 \right\}$.
Then, from the construction, $B_h$ is a classical solution of \eqref{E:main} in $\Rd\times (0,\infty) \setminus
W$. To check $B_h$ is a viscosity solution we just have to verify \eqref{E:is-visc-sol} at the points of $W$.
Note also that $B_h$ is $C^1$ everywhere and $DB_h = 0$ on $W$.

Let $P_0 = (x_0,t_0) \in W$ be arbitrary and assume $\phi$ touches $u$ from above at $P_0$, that is, it
satisfies \eqref{C:tfa}. Since $B_h$ is $C^1$, it must be that $\phi_t(P_0) = \del_t B_h(P_0)$ and $D\phi(P_0) =0$.
If $x_0 \ne 0$, then $A(|x_0|,t_0) = 0$ and $\phi_t(P_0) = 0$, otherwise, when $x_0 = 0$, we have $\phi_t(P_0) =
-\frac{c_h}{2h} t_0^{-\frac{2h+1}{2h}} R^{\frac{h+1}{h-1}}$. In any case we have
\[
    \phi_t(P_0) \leq 0 = -F_h(D^2\phi(P_0),D\phi(P_0))
\]
and \eqref{E:is-visc-sol} is satisfied.

Assume now $\phi$ touches $u$ from below at $P_0$. We claim that the function $B_h$ can not be touched from below
by a $C^2$ function at a point $(0,t_0)$. If this were the case we would have
\[\begin{aligned}
    c_h t^{-\oo{2h}} A(|x|,t)^{\frac{h}{h-1}} &- c_h t_0^{-\oo{2h}} R^{\frac{h+1}{h-1}} = B_h(x,t)-B_h(P_0)
        \geq \phi(x,t) - \phi(0,t_0) \\
        &\geq \phi_t(x,t) (t-t_0) + \oo{2} (D^2\phi(P_0)\, x)\cdot x + o(|x|^2) + o(t)
        \;\;\;\text{as $(x,t)\to P_0$}.
\end{aligned}\]
Taking $t=t_0$ and $x = \lambda e$ where $\lambda > 0$ and $e$ is any unit vector in $\Rd$ we obtain
\[
    c_h t_0^{-\oo{2h}} \left( A(\lambda,t_0)^{\frac{h}{h-1}} - R^{\frac{h+1}{h-1}} \right) \geq
        \frac{\lambda^2}{2} (D^2\phi(P_0)\, e)\cdot e + o(\lambda^2) \;\;\;\text{as $(x,t)\to P_0$}.
\]
Since
\[
    \lim_{\lambda\downarrow 0} \frac{A(\lambda,t_0)^{\frac{h}{h-1}} - R^{\frac{h+1}{h-1}}}{\lambda^2} = -\infty
\]
this leads to a contradiction and the claim is proved.

It follows that $P_0$ must satisfy $A(|x_0|,t_0) = 0$. At points where this occurs we have $\phi_t(P_0) =
\del_tB_h(P_0) = 0$ and $D\phi(P_0) = DB_h(P_0) = 0$, hence $\phi_t(P_0) + F_h(D^2\phi(P_0),D\phi(P_0)) = 0$ and
\eqref{E:is-visc-sol} is again satisfied.
\end{proof}
\begin{rmk}
Observe that for $h\in(1,2]$ the only points where $B_h$ is not twice differentiable are of the form $(0,t)$, hence
our proof can be simplified in these cases.
\end{rmk}

\subsection{Separation of variables}
Let us now look for solutions of the form
\[
    u(x,t) = T(t)X(r),
\]
where $r = |x|$. From equation \eqref{E:main} we must have
\[
    \frac{T'}{|T^{h-1}|T} = \frac{|X'|^{h-1}X''}{X} = -m
\]
for some $m\in\R$. Using the second symmetry in \eqref{E:sym}, we can rescale this equation to choose the
value of $m$ that better suits us. It is convenient to take $m = \pm \oo{h-1}$.

\subsubsection{Case $m = \oo{h-1}$} We can immediately integrate $T$ and obtain for an arbitrary $t_0\in\R$
\[
    T(t) = \oo{(t-t_0)^{\oo{h-1}}}.
\]
To integrate $X$ we use the change of variable
\[
    r(s) = \kappa \int_0^s \left(\sin(\sigma)\right)^{\alpha} \,d\sigma,
\]
where $\alpha = \frac{h-1}{h+1}$ and $\kappa^{h+1} = 2\alpha$ (a similar change of variable is performed
in \cite{AJK} in the case $h=3$). Since $r'(s) > 0$ for $s\in(0,\pi)$, this transformation is invertible from
$[0,\pi]$ to $[0,\overline{R}]$, where
\[
    \overline{R} = \kappa \int_0^\pi \left(\sin(\sigma)\right)^{\alpha} \,d\sigma.
\]
Hence, its inverse function $s:[0,\overline{R}]\to[0,\pi]$ is well defined, strictly increasing, $s(0) = 0$,
$s(\overline{R}) = \pi$ and $s'(r) = \kappa^{-1}\left( \sin\left(s(r)\right)\right)^{-\alpha}$. Let us define
$X_* :[0,\overline{R}] \to \R$ by
\[
    X_*(r) \bydef \cos(s(r)).
\]
We have
\[
    X_*'(r) = - \oo{\kappa}\sin^{1-\alpha}(s(r)), \quad
        X_*''(r) = -\frac{1-\alpha}{\kappa^2} \sin^{-2\alpha}(s(r))\cos(s(r)).
\]
Since $\alpha \in (0,1)$, the function $X_*$ is $C^1[0,\overline{R}]$ and $C^2(0,\overline{R})$. We can extend $X_*$ to
$\R^+_0$ by reflection,
\begin{equation}\label{E:X}
    X(r) = \begin{cases}
        X_*(r-2k\overline{R}) &\text{if $r\in[2k\overline{R},(2k+1)\overline{R})$ for $k\in\mathbb{N}_0$}, \\
        X_*(2k\overline{R}-r) &\text{if $r\in[(2k-1)\overline{R},2k\overline{R})$ for $k\in\mathbb{N}$}.
    \end{cases}
\end{equation}
Then $X$ is $2\overline{R}$-periodic $C^1$ function as is the function $|X'|^{h-1}X'$. Furthermore, the latter satisfies
\[
    \left[ |X'(r)|^{h-1}X'(r) \right]' = -\frac{1-\alpha}{\kappa^{h+1}}X(r) = -mX(r)
\]
even at the points where $X$ is not twice differentiable, that is, the points of the form $k\overline{R}$, with
$k\in\mathbb{N}_0$.
\begin{prop}
For any $t_0\in\R$ and any $r_0>0$, the function
\begin{equation}\label{E:friendly-giant}
    S(x,t;r_0,t_0) = \frac{X_{r_0}(x)}{(t-t_0)^{\oo{h-1}}},
\end{equation}
where $X_r(x) = \left( \frac{2r}{\overline{R}} \right)^{-\frac{h+1}{h-1}} X\left(\frac{\overline{R}|x|}{2r}\right)$,
with $X$ given by \eqref{E:X}, is a viscosity solution of \eqref{E:main} in $\Rd\times(t_0,\infty)$. In particular,
the restriction of $S$ to $Q_{r_0,t_0}(0) \bydef B_{r_0}(0)\times(t_0,\infty)$ is a positive
solution of the homogeneous Dirichlet problem in this set.
\end{prop}
\begin{proof}
Using the $x$-symmetry in \eqref{E:sym}, it is enough to consider $r_0 = \overline{R}/2$.
From the properties of $X$ deduced above, it is clear that $S$ is a classical, and consequently a viscosity solution
of \eqref{E:main} off the set $\Sigma \bydef \{ (x,t) \mid |x| = k\overline{R}, \; k\in\mathbb{N}_0 \}$. Moreover,
$S$ is $C^1$ in the whole space with $S_t(x,t) = -(t-t_0)^{-\frac{h}{h-1}} X(|x|)$ and $DS(x,t) =
(t-t_0)^{-\oo{h-1}}X'(|x|) \frac{x}{|x|}$.

We need to check that \eqref{E:is-visc-sol} is satisfied on $\Sigma$. To this effect, let $P_* = (x_*,t_*) \in \Sigma$,
$t_* > t_0$, and suppose $\phi \in C^{2,1}$ touches $S$ from above at $P_*$. As in the proof of
Proposition~\ref{P:barenblatt-is-visc}, since $X''(r) = +\infty$ at the points of the form $r = (2k+1)\overline{R}$,
$k\in\mathbb{N}_0$, we see that $|x_*|$ can not be of this form, hence $|x_*| = 2k\overline{R}$ for some
$k\in\mathbb{N}_0$. Now we have $D\phi(P_*) = DS(P_*) = 0$ and
\[
    \phi_t(P_*) + F(D^2\phi(P_*)) = -(t_*-t_0)^{-\oo{h-1}} < 0,
\]
which is \eqref{E:is-visc-sol}. The verification of the other half of Definition~\ref{D:visc-sol} is similar.
\end{proof}

\subsubsection{Case $m= -\oo{h-1}$}
In this case we obtain a blow-up solution with
\[
    T(t) = (t_0 - t)^{-\oo{h-1}},
\]
for $t < t_0$, and
\[
    X(r) = c_h r^{\frac{h+1}{h-1}},
\]
with $c_h$ as in \eqref{E:barenblatt}. In fact, we will need a slightly more general solution.
\begin{prop}\label{P:blow-up-sol}
For any $r_0,t_0\in\R$, the function
\[
    V(x,t;r_0,t_0) = \frac{c_h\left[|x|-r_0\right]]_+^{\frac{h+1}{h-1}}}{(t_0-t)^{\oo{h-1}}}
\]
is a viscosity solution of \eqref{E:main} in $\Rd\times(-\infty,t_0)$.
\end{prop}
The proof is similar to the proof of the previous proposition, having as the difficult points the sphere $\{|x| = r_0\}$
when $r_0 > 0$.

\subsection{Traveling waves}
Let us finally look for solutions of the form
\[
    u(x,t) = f(x_1 - ct).
\]
The function $f(\eta)$ must satisfy
\[
    |f'|^{h-3} f' f'' = -c.
\]
Taking the constants of integration equal to zero (we can recover them using the symmetries in \eqref{E:sym}), and
ignoring the stationary solutions ($c=0$), we integrate twice to obtain
\[
    f(\eta) = \oo{|c| h} \left( - c(h-1)\eta \right)^{\frac{h}{h-1}}.
\]
Therefore, we obtain
\[
    u(x,t) = \frac{d_h}{|c|} \left( c^2 t - cx_1 \right)^{\frac{h}{h-1}},
\]
with $d_h = \frac{(h-1)^{\frac{h}{h-1}}}{h}$, valid for $\frac{x_1}{c} \leq t$. Since this function is $C^1$ up to
the hyperplane $x_1 = ct$, with $u_t$ and $Du$ vanishing there, we can proceed as in the two previous propositions
to get the following one.
\begin{prop}\label{P:travel-wave-sol}
For any unitary vector $\nu\in\Rd$ and $c\in\R$, the function
\[
    T(x,t;\nu,c) = \frac{d_h}{|c|} \left[ c^2 t - cx\cdot\nu \right]_+^{\frac{h}{h-1}}
\]
is a viscosity solution of \eqref{E:main} in $\Rd\times\R$.
\end{prop}

\section{Existence}
The main purpose of this section is to prove a solution to \eqref{E:main}--\eqref{E:main-bc} exists. However, as we
mention in the introduction, and without much additional cost, it will useful to allow a zeroth-order term as in
\eqref{E:gen-main}. With $\eps \geq 0$ and $\delta > 0$, we start with the approximation
\begin{equation}\label{E:app-e-d}\begin{cases}
    u_t = \mathcal{L}^{\eps,\delta}(u) \bydef A^{\eps,\delta}(Du) : D^2 u + H(u) &\quad \text{in $Q$} \\
    u(P) = g^{\eps,\delta}(P), &\quad \text{on $\del_p Q$}.
\end{cases}\end{equation}
where the $d\times d$ matrix $A^{\eps,\delta}(p)$ is given by
\[
    A^{\eps,\delta}(p) = \eps I + (|p|^2 + \delta^2)^{(h-3)/2} p \otimes p.
\]
The proofs of the estimates for the approximate problem \eqref{E:app-e-d} are very similar to the corresponding
proofs of existence in \cite{PV} or \cite{JK}. For clarity we will present them schematically (see also
\cite{AS}, where a slightly different approximation is used).

\subsection{Lipschitz estimate in time}
We start with Lipschitz regularity in $t$ and then prove the regularity in $x$.

\begin{theorem}\label{T:lip-t0}
Suppose $g\in C^2({\del_p Q})$ and $u = u^{\eps,\delta}$ is a smooth solution of \eqref{E:app-e-d}.
Then there exists $K_1>0$ depending only on $\|D^2g\|_\infty$, $\|Dg\|_\infty$, $\|g\|_\infty$ and
$\|g_t\|_{\infty}$ such that
\[
    |u(x,t) - g(x,0)| \leq K_1 t
\]
for any $(x,t)\in Q$. If $g$ is only continuous in $x$ and bounded in $t$ then the modulus of continuity of $u$
on $\Omega\times[0,t_*]$ (for small $t_*$) can be estimated in terms of $\|g\|_\infty$ and the modulus of
continuity of $g_0\bydef g\big|_{\del_p Q}$ in $x$.
\end{theorem}
\begin{proof}
Assume for the moment $g \in C^2$ and $H\equiv 0$. Let $\lambda > 0$ and $v^{\pm}(x,t) = g_0(x) \pm \lambda t$. Then,
if we choose $\lambda$ large enough, $v^{+}$, respectively $v^{-}$, becomes a super-, respectively subsolution of
\eqref{E:app-e-d} which lies above, respectively below $u$ on $\del_p Q$. Therefore, by the classical comparison (see
for example \cite{LSU}),
\[
    |u(x,t) - g_0(x)| \leq K_0 t
\]
for $t\in[0,T]$ where $K_0$ is a constant depending only on the stated norms of $g$.

If $H\not\equiv 0$, then $v^{\pm}$ is still a super/subsolution of \eqref{E:app-e-d} in $\Omega\times(0,T_*)$, above/below
$g$ on $\del_p Q$, if we further choose $\lambda$ to satisfy $\lambda > 2M\|g\|_{\infty}$ and $T_* < (2M)^{-1}$.
Therefore the estimate is valid on this time interval. Then we can iterate the argument until we cover the whole interval
$[0,T]$.

Assume now $g$ is only continuous in $x$ and bounded in $t$ and let $\omega_0$ be the modulus of continuity of $g_0$.
Let us fix a point $x_0 \in \Omega$ and $0 < \rho < \min(\dist(x_0,\del \Omega),2\sqrt{\|g\|_\infty})$. Let us also
define
\[
    g^{\pm}(x,t) = g_0(x_0) \pm \omega_0(\rho) \pm \frac{2\|g\|_\infty}{\rho^2}|x-x_0|^2.
\]
It is easy to see that $g^- \leq g \leq g^+$ on $\Gamma$ and thus, again from the comparison principle,
$u^- \leq f \leq u^+$, where $u^\pm$ is the solution of \eqref{E:app-e-d} with initial and boundary condition
$g^\pm$. Since $g^\pm$ are in $C^2(\Rd\times\R)$, we can use the above estimate to conclude that
\[
    |u^\pm(x_0,t) - g_0^\pm(x_0)| \leq K^{\pm}_0 t
\]
where $K^{\pm}_0$ depends on $\|g\|_\infty$ and $\rho$. Therefore,
\[\begin{aligned}
    |u&(x_0,t) - g_0(x_0)| 
     \leq 2 K^+_0 t + \frac{3}{2} \omega_0(\rho).
\end{aligned}\]
This inequality concludes the proof.
\end{proof}

Using \eqref{T:lip-t0} we obtain the full Lipschitz estimate in time.

\begin{theorem}\label{T:lip-t}
If $u$ is a solution of \eqref{E:app-e-d} in $Q$ and $g \in C^2(\overline{Q})$, then there exists $K_2 > 0$
depending only on $\|D^2g\|_\infty$, $\|Dg\|_\infty$, $\|g\|_\infty$ and $\|g_t\|_\infty$ such that
\[
    |u(x,t) - u(x,s)| \leq K_2 |t - s|
\]
for every $x\in \Omega$, $t,s\in (0,T)$. If $g$ is merely continuous, we can estimate the modulus of continuity of
$u$ on $Q$ in terms of $\|g_0\|_\infty$ and the modulus of continuity of $g_0$.
\end{theorem}
\begin{proof}
Taking $\tau >0$ and
\[
    \hat{u}(x,t) \bydef u(x, t+\tau),
\]
using the Theorem~\ref{T:lip-t0} it is immediate to get
\[
    |u(x,t) - \hat{u}(x,t)| \leq K_2 t
\]
in $\Omega\times [0,T-\tau]$. The case when $g$ is only continuous is done as in the previous proof.
\end{proof}

\subsection{H\"older continuity in space}

\begin{theorem}\label{T:holder-bdry}
Let $u$ be the solution of \eqref{E:app-e-d} with $g\in C^2(Q)\cap \mathrm{Lip}(\overline{Q})$.
There exist $\alpha\in(0,1)$ and $K_3 > 1$, depending only on, $\|g\|_\infty$,
$\|Dg\|_\infty$, $\|g_t\|_\infty$ and $\alpha$, such that for every $\eps$ and $\delta$ sufficiently small and
for every $P_0 = (x_0,t_0) \in \Gamma$ and $x\in \Omega$ with $|x-x_0| \leq 1$ we have
\[
    |u(x,t_0) - g(x_0,t_0)| \leq K_3 |x-x_0|^\alpha.
\]
\end{theorem}
\begin{proof}
Let us define
\[
    v^+(x,t) = g(x_0,t_0) + K_*|x-x_0|^\alpha + \lambda\left( t_0-t \right),
\]
where $K^* \geq 1$ and $\lambda >0$ are constants which we will choose in such a way as to make $v^+$
a supersolution lying above $u$ on the appropriate domain. Let us take $x \in \Omega\cap B_{1}(x_0)$
and $t \in (0,t_0)$. An easy computation, using the fact that $|x-x_0| \leq 1$, yields
\[
    v^+_t - \mathcal{L}^{\eps,\delta} v^+ 
    \geq -\lambda + \frac{K_*\alpha(1-\alpha)}{|x-x_0|^{2-\alpha}} \left( f\left(
        \frac{\alpha K_*}{|x-x_0|^{1-\alpha}} \right) - \eps\frac{d-1}{1-\alpha} \right)
        - M( \|g\|_{\infty} + K_* + T\lambda),
\]
where
\[
    f(z) \bydef z^2 \left( \delta^2 + z^2 \right)^{\frac{h-3}{2}}.
\]
Since $|x-x_0| \leq 1$, we need only consider $z \geq \alpha K_*$. If $h \geq 3$, then $f(z) \geq z^{h-1}$ and when
$1 < h < 3$ we have $f(z) \geq 2^{(h-3)/2} z^{h-1} \geq z^{h-1}$ provided $z > \delta$. In any case it holds
\[\begin{aligned}
    v^+_t - \mathcal{L}^{\eps,\delta} v^+ &\geq -\lambda + \frac{K_*\alpha(1-\alpha)}{|x-x_0|^{2-\alpha}}
        \left( \alpha^{h-1} K_*^{h-1} - \eps \frac{d-1}{1-\alpha} \right) - M(\|g\|_\infty +K_* + T\lambda) > 0
\end{aligned}\]
whenever $\eps,\delta < 1 < \alpha K_*$,
\[
    \alpha K_* > \left( \frac{\eps(d-1)}{1-\alpha} \right)^{\oo{h-1}}
        \quad \text{and} \quad
    M\|g\|_{\infty} + (1+MT)\lambda < (1-\alpha) \alpha^h K_*^h - (\eps(d-1) + M)\alpha K_*.
\]

We want to have $v^+ > u$ on $Q^* = (\Omega\cap B_{1}(x_0)) \times (t_0 - t_*,t_0)$, where
we take $t_* \bydef \min\{1,t_0\}$. Let $P = (x,t) \in \del_p Q^*$. Let us first assume $P$ is on
the lateral boundary of $Q^*$. Since $u=g$ on $\del \Omega$ and $|x-x_0| <1$,
\[\begin{aligned}
    u(P) &\leq u(P_0) + \|Dg\|_{\infty} |x-x_0| + \|g_t\|_{\infty} (t_0 -t) \\
    & \leq g(P_0) + K_*|x-x_0|^\alpha + \lambda(t_0 - t) = v^+(P),
\end{aligned}\]
provided $K_* \geq \|Dg\|_{\infty}$ and $\lambda \geq \|g_t\|_{\infty}$.
If, on the other hand, $x\in \Omega\cap\del B_{1}(x_0)$, then, using comparison,
\[
    u(P) \leq \|g\|_{\infty} \leq u(P_0) + K_* + \lambda(t_0 - t)
        \leq v^+(P),
\]
provided $K_* \geq \|g\|_\infty$.

We consider now the case when $P$ is on the bottom of the cylinder $Q^*$. Let us first assume
$x\in \Omega\cap B_{1}(x_0)$ and $t = t_0 -1$. In this
case, again using comparison, we get
\[
    u(P) \leq \|g\|_{\infty} \leq u(P_0) + K_*|x-x_0|^\alpha + \lambda = v^+(P),
\]
as long as $\lambda \geq \|g\|_{\infty}$. Finally, when $t_0 <1$, and hence
$Q^* = (\Omega\cap B_{1}) \times (0,t_0)$, we have that $u = g$ on the bottom, therefore
\[\begin{aligned}
    u(P) &= u(x,0) = g(x,0) \leq g(x_0,t_0) + \|Dg\|_{\infty} |x-x_0| + \|g_t\|_{\infty} t_0 \\
    &\leq g(P_0) + K_* |x-x_0|^\alpha +\lambda t_0 = v^+(P),
\end{aligned}\]
provided, once again, $K_* \geq \|Dg\|_{\infty}$ and $\lambda \geq \|g_t\|_{\infty}$.

To summarize, we have $v^+ \geq u$ on $\del_p Q^*$, and hence, by comparison, $v^+ \geq u$ in
$Q^*$, if $\eps,\delta < 1$,
\[\begin{aligned}
    &K_*\alpha \geq \max\left\{ 1,\, \|Dg\|_\infty, \, \|g\|_\infty,\, \left( \frac{\eps(d-1)}{1-\alpha}
        \right)^{\oo{h-1}} \right\}, \\
    &M\|g\|_{\infty} + (1+MT)\lambda < (1-\alpha) \alpha^h K_*^h - (\eps(d-1) + M)\alpha K_*
        \quad \text{and} \\
    &\lambda \geq \max\left\{\|g_t\|_\infty,\,\|g\|_\infty \right\}.
\end{aligned}\]
Since $h>1$ we can choose the constants $K_*$ and $\lambda$ to satisfy these inequalities. We obtain
\[
    u(x,t_0) - g(x_0,t_0) \leq v^+(x,t_0) - g(P_0) = K_*|x-x_0|^\alpha.
\]
Using the barrier $v^-\bydef g(P_0) - K_*|x-x_0|^\alpha + \lambda(t - t_0)$ we get the reverse
inequality,
\[
    u(x,t_0) - g(x_0,t_0) \geq - K_*|x-x_0|^\alpha.
\]
\end{proof}

We can extend the estimate to the interior of the domain. We will use the following notation for convenience.
For $z\in \Rd$, we define $\Omega_z = z + \Omega = \{x+z\mid x\in \Omega \}$ and for $r>0$, $\Omega_r = \{ x \in \Omega \mid \dist(x,
\del \Omega) \geq r \}$. Given $x,y\in\Rd$ we define the closed segment $[x,y] \bydef \{ \theta y + (1-\theta) x
\mid 0\leq \theta \leq 1 \}$. The semi-open and open segments $[x,y)$, $(x,y]$ and $(x,y)$ are defined
analogously.

\begin{theorem}\label{T:holder-x}
The conclusion of Theorem~\ref{T:holder-bdry} is valid in the interior of $Q$, that is, there exists
$K_4$, depending only on $\|g\|_{\infty}$, $\|Dg\|_{\infty}$ and $\|g_t\|_{\infty}$, such that for
$\eps$ and $\delta$ sufficiently small and for every $x,y\in \Omega$
\[
    |u(x,t) - u(y,t)| \leq K_4|x-y|^ \alpha.
\]
\end{theorem}
\begin{proof}
To simplify notation, in this proof we omit the dependence of $u$ in $t$, since time does not play any role.
Take a vector $z\in B_{1}(0)$ and let $V = \Omega \cap \Omega_z$. Let us define
$u_z(x) \bydef u(x-z)$. From Theorem~\ref{T:holder-bdry} we have that $|u(x)-u_z(x)| \leq K_3|z|^\alpha$ on
$\del V$ ($x\in \del V$ implies that $x\in \del \Omega$ or $x-z\in \del \Omega$). Hence, using the comparison
principle we have that $u_z(x) - K_3|z|^\alpha \leq u(x) \leq u_z(x) + K_3|z|^\alpha$ for $x\in V$. This means
that whenever $x,y\in \Omega_{x-y}\cap  \Omega$ or $x,y\in \Omega_{y-x} \cap \Omega$, with $|x-y| \leq \rho_3$,
we have $|u(x) - u(y)| \leq K_3 |x-y|^\alpha$. In particular, the same is true whenever $x,y\in \Omega_{|x-y|}$.

When $|x-y| > 1$, using the comparison principle we obtain the conclusion of the theorem taking
$K_4 = 2\|g\|_\infty$. Let us therefore assume that $|x-y| \leq 1$ and  $x-y\notin \Omega_{|x-y|}$.
Let us first further assume that $[x,y] \subset \Omega$. In this case we can take the two segments $[x,w]$ and
$[w,y]$, where $w = (x+y)/2$ is the midpoint of $[x,y]$, let $z = y-w$ and note that $w,y\in \Omega_z \cap \Omega$ and
$x,w\in \Omega_{-z}\cap \Omega$. Hence, from the first step of this proof, we have
\[\begin{aligned}
    |u(x) - u(y)| &\leq |u(x) - u(w)| + |u(w) - u(y)| \leq K_3(|x-w|^\alpha + |w-y|^\alpha) \\
    &\leq 2^{1-\alpha}K_3 |x-y|^\alpha.
\end{aligned}\]

If the segment $[x,y]$ is not completely in $\Omega$, then we can certainly find $w_1, w_2 \in \del \Omega \cap [x,y]$
(not necessarily different) such that $[x,w_1) \in \Omega$ and $(w_2,y] \in \Omega$. In this case we can apply
Theorem~\ref{T:holder-bdry} directly to get $|u(x) - u(w_1)| \leq K_3 |x-w_1|^\alpha$ and $|u(w_2) - u(y)| \leq
K_3 |w_2-y|^\alpha$. Since $|u(w_1) - u(w_2)| \leq \|Dg\||w_1 - w_2| \leq K_3|w_1 - w_2|^\alpha$, we easily get
the result with $K_4 = 4^{1-\alpha} K_3$. This finishes the proof.
\end{proof}

\subsection{Lipschitz estimate in space ($\eps=0$)}

As in \cite{PV}, we obtain Lipschitz regularity in space when we take $\eps = 0$ in the approximation \eqref{E:app-e-d}.

\begin{theorem}\label{T:lip-bdry}
Let $g \in \mathrm{Lip}(\overline{Q})$ and suppose $u$ is a viscosity solution of \eqref{E:app-e-d} with $\eps =0$.
There exist a constant $K_5$, depending only on $\|g\|_{\infty}$, $\|Dg\|_{\infty}$ and
$\|g_t\|_{\infty}$ (independent of $\delta\in(0,1)$), such that for every $P_0 = (x_0,t_0) \in \del
\Omega\times (0,T)$ and $x\in \Omega\cap B_{1}(x_0)$ we have
\[
    |u(x,t_0) - g(x_0,t_0)| \leq K_5 |x-x_0|.
\]
Furthermore, if $g$ is only continuous, then the modulus of continuity of $u$ can be estimated in terms of
$\|g\|_{\infty}$ and the modulus of continuity of $g$.
\end{theorem}

\begin{proof}
Let $K_*$, $L_*$ and $\lambda$ be positive constants and define
\[
    v^+(x,t) = g(P_0) + L_*|x-x_0| - K_*|x-x_0|^2 + \lambda(t_0 - t).
\]
We will check $v^+$ is a viscosity strict supersolution on $Q^*$ ($Q^*$ defined as in the proof of
Theorem~\ref{T:holder-bdry}) above $g$ for an appropriate choice of constants $L_*$, $K_*$ and $\lambda$.
Observe that
\[
    v^+_t(x,t) - \mathcal{L}^{0,\delta} v^+(x,t) = -\lambda + 2 K_* f(|L_* - 2K_*|x-x_0||) - H(v^+(x,t)),
\]
where, as in the proof of Theorem~\ref{T:holder-bdry}, $f(z) = (z^2 + \delta^2)^{(h-3)/2} z^2$. Recall from that proof
that $f(z) \geq z^{h-1}$ for $z \geq 1$ and for all $h>1$. Hence, if we choose $L_*$ and $K_*$ such that $L_* > 3K_*$ and
$2K_*^{h+1} -2MK_* > (1+MT)\lambda + M\|g\|_\infty$, the above inequality shows that $v^+$ is a supersolution.

We need to further choose the constants so that $v^+ \geq f$ on the parabolic boundary of $Q^*$. Let $P = (x,t)$
be a point in $\Gamma^* = \del_p Q^*$. If $x\in\del \Omega$, as before
\[\begin{aligned}
    u(P) &= g(P) \leq g(P_0) + \|Dg\|_{\infty} |x-x_0| + \|g_t\|_{\infty}(t_0-t) \\
    &\leq g(P_0) + (L_*-1)|x-x_0| + \lambda(t_0-t) < v^+(P),
\end{aligned}\]
provided $L_* \geq \|Dg\|_{\infty} + 1$, $K_* \geq 1$ and $\lambda >\|g_t\|_{\infty}$. If
$x\in \Omega\cap \del B_{1}(x_0)$, then
\[
    u(P) \leq \|g\|_{\infty} \leq g(P_0) + (L_* -1)|x-x_0| + \lambda(t_0-t) \leq v^+(P),
\]
provided $L_* \geq \|g\|_{\infty} + 1$ and yet again $K_* \geq 1$.

When $P = (x,t)$ is on the bottom of the cylinder $Q^*$, as before we consider two cases. When $t_0 \geq 1$,
$t = t_0-1$ and
\[
    u(P) \leq \|g\|_{\infty} \leq g(P_0) + \lambda < v^+(P)
\]
as long as $\lambda \geq \|g_t\|_{\infty}$ and $L_* \geq \max\{1,K_*\}$. On the other hand, if
$t_0 <1$, $t=0$ and hence
\[
    u(p) = g(P) \leq \|g\|_{\infty} \leq g(P_0) + \lambda < v^+(P)
\]
under the exact same conditions as for the previous formula.

Therefore, we have $u \leq v^+$ on $\Gamma^*$, and thus by comparison on $Q^*$, as long as we take
\[\begin{aligned}
    \lambda \geq \|g_t\|_{\infty}, \quad
    K_* \geq \max\left\{ 1, \sqrt{\lambda/2} \right\} \quad \text{and} \quad
    L_* \geq \max\left\{ 2, \, \|Dg\|_{\infty}+1, \, \|g\|_{\infty}+1,\, 3K_* \right\}.
\end{aligned}\]
Using once more the comparison principle, we have that for $x\in \Omega\cap B_{1}(x_0)$,
\[
    u(x,t_0) \leq v^+(x,t_0) \leq g(P_0) + L_*|x-x_0|.
\]
Using instead the barriers
\[
    v^-(x,t) = g(P_0) - L_*|x-x_0| + K_*|x-x_0|^2 + \lambda(t-t_0)
\]
we obtain the reverse inequality and, as a consequence, the Lipschitz estimate.

Let us finally merely assume that $g$ is continuous and let $\omega_g(\sigma)$ be a modulus of continuity
at $P_0$. More specifically, let $\omega_g$ be a continuous, decreasing function in $\sigma$ such that
$|g(P) - g(P_0)| \leq \omega_g(\sigma)$ whenever $\max\{{|x-x_0|},\, |t-t_0|\} \leq \sigma$. Let $\sigma \in (0,t_0)$
and define the smooth functions
\[
    g^{\pm}(x,t) \bydef g(x_0,0) \pm \omega_g(\sigma) \pm \frac{4\|g\|_{\infty}}{\sigma^2} |x-x_0|^2
        \pm \frac{2\|g\|_{\infty}}{\sigma} |t - t_0|.
\]
If $\max\{|x-x_0|,\, |t-t_0|\} \leq \sigma$, then
\[
    g^{-} (P) \leq g(P_0) - \omega_g(\sigma) \leq g(P) \leq g(P_0) + \omega_g(\sigma) \leq g^+(P),
\]
and if $\max\{|x-x_0|,\, |t-t_0|\} \geq \sigma$ then
\[
    g^-(P) \leq -\|g\|_{\infty} \leq g(P) \leq \|g\|_{\infty} \leq g^+(P).
\]
Therefore, if $u^\pm$ are the solutions of \eqref{E:app-e-d} with $\eps=0$ and initial data $g^\pm$, by comparison
$u^- \leq u \leq u^+$ on $Q$. Since $u^\pm$ are smooth we can apply the first part of the theorem to deduce that
\[
    |u^\pm(x,t_0) - g^\pm(P_0)| \leq K_5^+ |x-x_0|,
\]
where $K_5^+$ depends on $\|g\|_{\infty}$ and $\sigma$. From these inequalities we get
\[\begin{aligned}
    |u(x,t_0) &- g(P_0)| 
    \leq 2K_5^+|x-x_0| + \frac{3}{2}\omega_g(\sigma).
\end{aligned}\]
This finishes the proof.
\end{proof}

Our final estimate is the interior Lipschitz estimate.
\begin{theorem}\label{T:lip-x}
Let $g$ and $u$ be as in Theorem~\ref{T:lip-bdry}. For every $x,y\in \Omega$ and $t\in(0,T)$
\[
    |u(x,t) - u(y,t)| \leq K_5|x-y|,
\]
where $K_5$ is the constant given in that theorem.
If $g$ is only continuous, then the modulus of continuity of $x\mapsto u(x,t)$ can be estimated in terms
of $\|g\|_\infty$ and the modulus of continuity of $g$ in $x$.
\end{theorem}
\begin{proof}
The proof is similar to the proof of Theorem~\ref{T:holder-x}, but in this case it is easy to get
the optimal Lipschitz constant. Once again we omit the time dependence of $u$. Take $z\in\Rd$ such that
$|z| \leq \rho_5$. Define $V = \Omega \cap \Omega_z$ and let
$u_z(x) \bydef u(x-z)$. From previous theorem we know that $|u(x)-u_z(x)| \leq K_5 |z|$ on $\del V$. Using
comparison, we have that $u_z(x) - K_5|z| \leq u(x) \leq u_z(x) + K_5|z|$ in $V$. Therefore,
$|u(x)-u(y)| \leq K_5|x-y|$ if $x,y\in \Omega_{x-y}$, and in particular the same is true if $x,y \in \Omega_{|x-y|}$,
$\Omega_r = \{ x\in \Omega \mid \dist(x,\del \Omega) >r\}$.

Suppose now $x-y\notin \Omega_{|x-y|}$ and let us first assume that the whole segment $[x,y] = \{ z\in \Rd \mid z
= \theta y + (1-\theta) x, \; 0\leq\theta\leq 1 \}$ is in $\Omega$. Let us assume without loss of generality that
$\rho = \dist(x,\del \Omega) \leq \dist(y, \del \Omega)$. We can find points $x_i$, $0\leq i \leq n$ such that $x = x_0$,
$x_n = y$, $x_i \in [x_{i-1},x_{i+1}]$ ($1\leq i \leq n-1$), and $\rho_i = |x_i - x_{i-1}| \leq \rho$
($1\leq i \leq n$). Noting that $x_i,x_{i-1}\in \Omega_{x_i-x_{i-1}}$ we can use the previous step to conclude that
$|u(x_i) - u(x_{i-1})| \leq K_5|x_i - x_{i-1}|$, and hence
\[
    |u(x) - u(y)| \leq \sum_{i=1}^n |u(x_i) - u(x_{i-1})| \leq K_5 |x-y|.
\]

If, on the other hand $[x,y] \notin \Omega$, then we can find points $x_1, x_2 \in \del \Omega \cap [x,y]$ such that
$[x,x_1) \subset \Omega$ and $[x_2,y]\setminus\{x_2\} \subset \Omega$. We can further choose $w_1\in [x,x_1]$, with
$|w_1 - x_1| \leq 1$ and $w_2 \in [x_2,y]$, with $|w_2 - x_2| \leq 1$. Then we apply the above to
obtain $|u(x) - u(w_1)| \leq K_5 |x-w_1|$, $|u(w_2) - u(y)| \leq |w_2-y|$, while from the previous theorem,
$|u(w_i) - u(x_i)| \leq K_5 |w_i - x_i|$. Putting all these inequalities together gives the Lipschitz estimate
for this last case.

The proof of the statement with the modulus of continuity follows as in the proof of
Theorem~\ref{T:holder-x}.
\end{proof}

We finally prove Theorem~\ref{T:exist}. Existence is proved by piecing out the results in Theorems~\ref{T:lip-t},
\ref{T:holder-x} and \ref{T:lip-x}, and using the standard compactness arguments, as is done in
\cite[Theorem~1.1]{PV}. Uniqueness follows directly from the comparison principle, Theorem~\ref{T:comparison}.

\begin{proof}[Proof of Theorem~\ref{T:exist}]
The proof when $\Omega$ is bounded is similar to the proof of \cite[Theorem~1.1]{PV} and we just sketch it here.
Assume first $g\in C^2(Q)\cap \mathrm{Lip}(\overline{Q})$. The comparison principle and Theorems~\ref{T:lip-t}
and~\ref{T:holder-x} imply that the family of functions $\{u^{\eps,\delta}\}$ is uniformly bounded and
equicontinuous, therefore, for some sequence $\eps_k \to 0$, $u^{\eps_k,\delta} \to u^\delta$, which, by a
standard argument of viscosity solutions, is a solution of \eqref{E:app-e-d} with $eps=0$.
Then, using Theorems~\ref{T:lip-t} and~\ref{T:lip-x}, we can find a sequence $f^{\delta_k} \to u$. The stability
arguments for viscosity solutions work here to show that $u$ is a viscosity subsolution of \eqref{E:main}.

If $\Omega$ is not bounded we define $\Omega_R \bydef \Omega\cap B_r(0)$, $Q_R \bydef \Omega_R\times (0,T]$ and
$g_R: \Gamma_R\to\R$, $\Gamma_R = \del_p Q_R$, by $g_R(x,t) \bydef 0$ if $|x|=R$, $g_R(x,t) \bydef \chi_R(x)g(x,t)$
for $(x,t)\in\Gamma\cap Q_R$, where $\chi_R(x) = \chi(x/R)$ and $\chi\in C^\infty_{\mathrm{c}}(\Rd)$ satisfies
$\chi(x) = 1$ if $|x| \leq 1/2$, $\chi(x) = 0$ if $|x| \geq 1$. In $Q_R$ there exists a unique viscosity solution
with initial/boundary data $g_R$. From the assumptions on $g_R$, the estimates for $u_R$ and the stability of the
viscosity solutions we can let $R\to \infty$ to obtain the result.
\end{proof}

\section{Asymptotic behaviour in the whole space}

We consider in this section Cauchy problem \eqref{E:main-cauchy}. We first obtain a decay rate for the solutions of this
problem and then prove the asymptotic convergence.

\subsection{Decay rate}
Using the similarity solutions and the comparison principle we readily obtain the following estimates.
\begin{theorem}\label{T:cauchy-decay-rate}
Let $u$ be the unique viscosity solution of \eqref{E:main-cauchy} with $u_0\in C_{\mathrm{c}}(\Rd)$ not identically
zero, $u_0 \geq 0$. Then, there exist positive constants $c$ and $C$ depending only on $u_0$ such that for $t>0$
\begin{equation}\label{E:cauchy-bounds}
    c (1+t)^{-\oo{2h}} \leq \max_{x\in\Rd} |u(x,t)| \leq C (1+t)^{-\oo{2h}}.
\end{equation}
Moreover, the support of $u$ expands continuously at a rate of the order of $t^{\oo{2h}}$.
\end{theorem}
\begin{proof}
We can find similarity solutions $B_1$ and $B_2$ of the form \eqref{E:barenblatt} such that
\[
    \pm B_1(x,0) \leq u_0(x) \leq \pm B_2(x,0),
\]
which immediately gives \eqref{E:cauchy-bounds}. Indeed, if $x_o$ is a point where $u_0(x_o) \ne 0$, then, with $r$ sufficiently
small we can take $B_1(x,t) = B_{r,h}(x-x_0,t+1)$ and with $R$ sufficiently large $B_2(x,t) = B_{R,h}(x,t+1)$.

To prove that the support expands continuously, take a point $P_*$ on the boundary of the support of $u$ and the take a function
of the form $B_1$, as above, with $B_1 \leq u$, with center sufficiently close to $P_*$. Then, the comparison principle implies
that the support cannot jump ``inward'' toward the support of $B_1$. In the other direction, we can find a function $V_2$ of the
form given in Proposition~\ref{P:blow-up-sol} such that $u \leq V_2$ for a short time. Again the comparison principle implies that
the support can not jump ``outward'' beyond the support of $V_2$. The rate of expansion of the support of $u$ is controlled by
the rates of expansion for $B_1$ and $B_2$ above, that is, it has to be $t^{\oo{2h}}$.
\end{proof}

\subsection{Asymptotic behaviour. Proof of Theorem \ref{T:Cauchy}}

Step 1. The idea of the proof is the same as in the proof of \cite[Theorem~1.4]{PV}. First we let $B_1$ and $B_2$ sandwich $u$
as in the proof of Theorem~\ref{T:cauchy-decay-rate}. Then consider the family of rescaled solutions $u^\lambda$, $B_1^\lambda$
and $B_2^\lambda$, where for a function $f(x,t)$ we define $f^\lambda(x,t)$ by
\[
    f^\lambda(x,t) = \lambda^{\oo{2h}} f(\lambda^{\oo{2h}}x,\lambda t).
\]
Since $B_1$ and $B_2$ are invariant under this transformation, on any compact time interval $[t_1,t_2]$ with $0<t_1<t_2<\infty$
the family $u^\lambda$ is continuous, uniformly bounded, and supported on a uniform ball $B_{R_*}(0)$.

Step 2. Now we use Aleksandrov's principle, as explained for instance in \cite{CVW}, to show that
for a solution $u(x,t)$ with initial data $u_0(x)\geq 0$ supported in the ball $B_{R}(0)$ we have
for all $t\geq 0$ and all $r>R$
\[
    \inf_{|x|=r} u(x,t)= \max_{|x|=r+2R} u(x,t)
\]
Note that in doing this, we need to use the traveling wave solutions from Proposition~\ref{P:travel-wave-sol} to show that
the solutions are almost radial. If this is applied to the rescaled solutions, we get for all $|x|\geq R_\lambda =
R\,\lambda^{-1/(2h)}$
\[
    \inf_{|x|=r} u^\lambda(x,t)= \max_{|x|=r+2R_\lambda} u^\lambda(x,t).
\]

Step 3. We now fix $t=1$, $\lambda $ very large, so that $R_\lambda\leq \eps$ is very small, and define
\[
    \tilde u_1(r) \bydef \inf_{|x|=r} u^\lambda(x,1), \quad
    \tilde u_2(r) \bydef \max_{|x|=r} u^\lambda(x,1).
\]
We easily verify that $\tilde u_1(r),\, \tilde u_2(r)$ are nonnegative and radially symmetric functions,
both supported in the same ball $B_{R_*}(0)$, they are nonincreasing as functions of $r$ for $r\geq \eps$,
and we also have
\[
    \tilde u_2(r)\geq \tilde u_1(r)\geq \tilde u_1(r+\eps)
\]
for all $r\geq \eps$. It is then easy to verify that the 1-$d$ mass of $\tilde u_2(r) - \tilde u_1(r)$ is less than
$C\eps $.

Step 4. If $u_1(r,t)$ and $u_2(r,t)$ are the corresponding radial solutions of
the problem with initial data at $t=1$ given by $\tilde u_1(r)$ and $\tilde u_2(r)$, respectively, we have for all $t\geq 1$
\[
    u_1(r,t)\leq u^\lambda(x,t)\leq u_2(r,t)
\]
As with the convergence result for the 1-$d$ PME, cf.\ \cite[Theorem 18.1]{V}, the result follows.

\section{Asymptotic behaviour on bounded domains}\label{S:bdd-dom}

In this section we analyze the asymptotic behaviour of the solution of the homogeneous
Dirichlet problem \eqref{E:homo-Dirichlet}. As in the previous section, we first obtain a decay rate for the solutions of the
problem and then prove the asymptotic convergence.

\subsection{Decay rate}
We start by proving that the decay rate for the solutions of \eqref{E:homo-Dirichlet} is $t^{-\oo{h-1}}$, the same as the decay
rate of the friendly giants in \eqref{E:friendly-giant}. Of course, this is not a coincidence.

\begin{theorem}\label{T:Dirichlet-decay-rate}
Let $u_0$ satisfy the condition of Theorem~\ref{T:asymptotic-Dirichlet}, let $u$ be the unique solution of \eqref{E:homo-Dirichlet} and
assume $x_0$ is a point where $u_0(x_0) >0$. There exist positive
constants $t_0$, $r_1$ and $r_2$ such that, with $X_r$ as in \eqref{E:friendly-giant}, there holds for every $t>0$,
\[\begin{aligned}
    X_{r_1}(x-x_0) \leq (t+t_0)^{\oo{h-1}} u(x,t),& &&\quad \text{for $x\in B_{r_1}(x_0)$ and} \\
    t^{\oo{h-1}} u(x,t)& \leq X_{r_2}(x-x_0), &&\quad \text{for $x\in\Omega$.}
\end{aligned}\]
\end{theorem}
\begin{proof}
It it straightforward, under the conditions of the theorem, to find $r_1$ and $t_0$ such that $S(x-x_0,0;r_1,-t_0) \leq u_0(x)$
on the ball $B_{r_1}(x_0)$ and $r_2$ such that $\Omega\subset B_{r_2}(x_0)$. The result is then immediate from the comparison
principle.
\end{proof}

\subsection{Friendly giant in $\Omega$. Proof of Theorem \ref{T:asymptotic-Dirichlet}}
Let us now show there exists a friendly giant in $\Omega$. We will show that the asymptotic profile for arbitrary
(nonnegative) initial conditions is the profile of this friendly giant.

According to \eqref{E:sym}, we can rescale a solution $u$ of \eqref{E:homo-Dirichlet} by
\[
    u_\lambda(x,t) \bydef \lambda^{\oo{h-1}} u(x,\lambda t)
\]
and still obtain a solution of the same problem but with initial condition $u_{0\lambda}(x) = \lambda^{\oo{h-1}} u_0(x)$.
For $\lambda <1$, by comparison, we see that $u_\lambda(x,t) \leq u(x,t)$. According to \cite[Theorem~2.3]{BC}, we deduce
that
\begin{equation}\label{E:BC-estimate}
    u(x,t+\tau)-u(x,t) \geq -\left[ 1 - \left( \frac{t}{t+\tau} \right)^{\oo{h-1}}\right] u(x,t)
\end{equation}
for $(x,t)\in Q$ and $\tau>0$ such that $t+\tau<T$. Let us now consider the following rescaling of $u$,
\[
    v(x,s) \bydef (h-1)^{\oo{h-1}} e^s u \left(x,e^{(h-1) s} \right).
\]
It is easy to see that $v$ is a viscosity solution of
\[\begin{cases}
    v_s - \lapih v = v &\text{in $\Omega\times (1,\infty]$,} \\
    v(x,0) = (h-1)^{\oo{h-1}} u(x,1) &\text{for $x\in\Omega$,} \\
    v(x,s) = 0 &\text{for $x\in\del\Omega$, $s>1$.}
\end{cases}\]
The estimates for $u$ in Theorem~\ref{T:Dirichlet-decay-rate} and in \eqref{E:BC-estimate} imply the following estimates
for $v$:
\begin{equation}\label{E:v-estimates}
    v(x,t) \leq M,\qquad v(x,s+h) - v(x,s) \geq 0.
\end{equation}
Using our similarity solutions \eqref{E:barenblatt}, the proof of \cite[Lemma~3.2]{LS} adapts step by step,
\textit{mutatis mutandis}, and we have that $v$ eventually becomes positive on any compact subset of $\Omega$.
\begin{lemma}[\mbox{\cite[Lemma~3.2]{LS}}]\label{L:v-pos}
For any compact set $K \subset \Omega$ there exist $s_K$ and $m_K$ such that
\[
    v(x,s) \geq m_K \qquad \text{on $K\times[s_K,\infty)$}.
\]
\end{lemma}
Using this and the estimates \eqref{E:v-estimates}, we deduce that there exists a lower semicontinuous function
$G_\Omega:\Omega\to \R$ such that $G_\Omega(x) = 0$ on $\del\Omega$ and
\[
    \lim_{s\to\infty} v(x,s) = G_\Omega(x).
\]
Since $v$ is a viscosity solution of \eqref{E:gen-main}, from Theorem~\ref{T:exist} we see that we can control
the modulus of continuity of $v$ and so $G_\Omega$ must be continuous. In terms of $u$, this means
\[
    \lim_{t\to\infty} t^{\oo{h-1}} u(x,t) = F_\Omega(x) \bydef (h-1)^{-\oo{h-1}}  G_\Omega(x).
\]
\begin{theorem}\label{T:evp}
The function $G_\Omega$ is a positive viscosity solution of the eigenvalue problem
\[\begin{aligned}
    -\lapih G_\Omega &= G_\Omega \qquad \text{in $\Omega$}, \\
    G_\Omega(x) &= 0, \qquad \text{for $x\in\del\Omega$}.
\end{aligned}\]
\end{theorem}
The proof is a simplification of \cite[Theorem~7.3]{PV}. With this we finish the proof of
Theorem~\ref{T:asymptotic-Dirichlet}.

\section{\bf Comments and extensions}

\noindent $\bullet$  The restriction of nonnegativity in the study of asymptotic behaviour can be somewhat weakened
by observation that the equation is invariant under constant displacement of the $u$-variable. Hence, we can assume
that the initial data are compactly suported perturbations of the level $u=c$, $c$ constant. On the other hand, the restriction to compactly supported data can be weakened, but we do not know how to find a reasonably wider class, much
less an optimal class.

\noindent $\bullet$ The solutions of the form $u(x,t)=t^{-\oo{h-1}} X_r(x)$ described in \eqref{E:friendly-giant} are interesting examples of radial solutions which decay in time but not in $|x|$; they oscillate radially in a sine-like fashion. This behaviour is distinct from the several-dimensional heat equation, where it is known that such  solutions decay like a Bessel function as $|x|$ goes to infinity. The restriction of our functions to appropriately chosen balls also gives an example of an asymptotic profile which changes sign as many times as desired. This naturally leads to the question of classifying asymptotic profiles for  sign-changing initial data. We do not know if it is possible to obtain nonradial stable profiles on a ball.

\noindent $\bullet$ In both settings, Dirichlet problem in a bounded domain with zero boundary data and Cauchy Problem
in the whole data, we generate a semigroup enjoying the comparison principle. Since the operator is $h$ homogeneous with $h\ne 1$, the homogeneity estimate of B\'enilan-Crandall \cite{BC} applies and we have for all nonnegative solutions
$$
u_t\ge -\frac1{(h-1)t}\,u
$$
in the sense of distributions. This estimate could be important in deeper studies.

\noindent $\bullet$ The ``zeroth-order'' large time behaviour for the Dirichlet problem with nonhomogeneous boundary conditions is considered in \cite{AJK}. We do not consider this case here, but the techniques used there should work for our class of operators.

\noindent $\bullet$ We are not exploring the cases $h=1$ which is a well-known equation, or the cases $h<1$ which must
have novel properties.

\

\section*{Acknowledgments}
Both authors are partially supported by the Spanish Project MTM2008-06326. The first author is also supported
by the Portuguese Project PTDC-MAT-098060-2008


\end{document}